\documentclass[10pt,a4paper]{article}
\usepackage{amsmath}
\usepackage{amssymb}
\usepackage{amsthm}
\usepackage{amsfonts}
\usepackage{enumerate}
\usepackage{pstricks}
\usepackage{pst-plot}
\usepackage{pst-poly}
\usepackage{graphics} %cx na konci pro PDF
\usepackage{graphicx}
\usepackage{pgfpages}
\usepackage{url}
\usepackage{booktabs}
\usepackage{multirow}
\usepackage{verbatim}
\usepackage[margin=1.23in]{geometry}
\newcommand{\tluste}[1]{\mbox{\mathversion{bold}$ #1 $}}
\newcommand{\vr}[1]{{{#1}}}

\newcommand{\mace}[1]{{{#1}}}

\newcommand{\omace}[1]{\mbox{$\overline{\mace{#1}}$}} 
\newcommand{\umace}[1]{\mbox{$\underline{\mace{#1}}$}} 
\newcommand{\imace}[1]{\mbox{$\tluste{#1}$}}

\def\Mid#1{#1_c}
\def\Rad#1{#1_\Delta}
\def\MID#1{\mathop{\mathrm{mid}}#1}	%center of interval, written form
\def\RAD#1{\mathop{\mathrm{rad}}#1}	%radius of interval, written form
\newcommand{\ovr}[1]{\mbox{$\overline{\vr{#1}}$}} 
\newcommand{\uvr}[1]{\mbox{$\underline{\vr{#1}}$}}

\newcommand{\ivr}[1]{\mbox{$\tluste{#1}$}} 

\newcommand{\R}[0]{{\mathbb{R}}}

\newcommand{\IR}[0]{{\mathbb{IR}}}

\newcommand{\mmid}[0]{;\,}		%pouzivejte v definici mnozin!

\def\clqq{``}
\def\crqq{''}
\def\quo#1{\clqq{}#1\crqq{}}  % snadny zapis ang. uvozovek

\DeclareMathOperator{\sgn}{sgn}	
\DeclareMathOperator{\diag}{diag}	%diagonal matrix

\def\nref#1{$(\ref{#1})$}

\newtheorem{theorem}{Theorem}

\newtheorem{proposition}{Proposition}
\newtheorem{lemma}{Lemma}
\newtheorem{corollary}{Corollary}

\theoremstyle{definition}

\newtheorem{example}{Example}
\begin{document}

\title{AE solutions and AE solvability to general interval linear systems}

%\author{Milan Hlad\'{i}k \and Evgenija D. Popova}
\author{
  Milan Hlad\'{i}k\footnote{
Charles University, Faculty  of  Mathematics  and  Physics,
Department of Applied Mathematics, 
Malostransk\'e n\'am.~25, 11800, Prague, Czech Republic, 
e-mail: \texttt{milan.hladik@matfyz.cz}
}
%\footnote{University of Economics, Faculty of Informatics and Statistics,
%n\'{a}m. W.~Churchilla 4, 13067, Prague, Czech Republic}
}

\date{\today}
\maketitle

\begin{abstract}
We consider linear systems of equations and inequalities with coefficients varying inside given intervals. We define their solutions (so called AE solutions) and solvability (so called AE solvability) by using forall-exists quantification of interval parameters. We present an explicit description of the AE solutions, and discuss complexity issues as well. For AE solvability, we propose a sufficient condition only, but for a specific sub-class of problems, a complete characterization is developed. Moreover, we investigate inequality systems for which AE solvability is equivalent to existence of an AE solution.
\end{abstract}

%\textbf{Keywords:}\textit{ Interval matrix, interval analysis, eigenvalue, eigenvalue bounds.}

%%%%%%%%%%%%%%%%%%%%%%%%%%%%%%%%%%%%%%%%%%%%%%%%%%%%%%%%%%%%%%% 
% INTRODUCTION
%%%%%%%%%%%%%%%%%%%%%%%%%%%%%%%%%%%%%%%%%%%%%%%%%%%%%%%%%%%%%%% 
\section{Introduction}

Interval linear systems appear in many situations. The basic problem of solving interval linear equations \cite{Fie2006,MooKea2009,Neu1990,Rum2010,Sha2002} is important in solving and verifying real-valued linear and nonlinear systems, and in solving engineering problems with uncertain data. Interval linear inequalities emerge in global optimization when linearization techniques are used \cite{HlaHor2014a} and in mathematical programming when dealing with uncertainty \cite{AllNeh2013,Hla2012a,Hla2014a,Hla2014:a}.

Traditionally, a solution of an interval system is defined as a solution for some realization of intervals. In order to model robustness in interval equations solving, generalized concepts of solutions using quantifications appeared. The commonly used one is an AE solution \cite{Gol2005,GolChab2006,Pop2012,PopHla2013,Sha2002}, characterized by $\forall\exists$-quantification of interval parameters. 

To the best of our knowledge, the concept of AE solutions has not been utilized for interval inequalities yet. There are, however, many special cases studied. In interval linear programming, for example, the concept of quantified solutions was recently introduced in \cite{LiLuo2013,LuoLi2013a,LuoLi2014a}.

In this paper, we study AE solutions for general interval linear systems of equations, inequalities or both. Its direct applicability is in interval linear programming to characterize various models of robust solutions.

\paragraph*{Notation.} 
%Let us introduce some notation. 
%The $k$th row of a matrix $A$ is denoted as $A_{k*}$, and $\diag(s)$ stands for the diagonal matrix with entries given by $s$. The sign of a real $r$ is defined as $\sgn(r)=1$ if $r\geq0$ and $\sgn(r)=-1$ otherwise; for vectors the sign is meant entrywise.

The sign of a real $r$ is defined as $\sgn(r)=1$ if $r\geq0$ and $\sgn(r)=-1$ otherwise; for vectors the sign is meant entrywise. Next, $\diag(s)$ stands for the diagonal matrix with entries given by $s$.
%Next, $\diag(s)$ stands for the diagonal matrix with entries given by $s$, and $e=(1,\dots,1)^T$ for the vector of ones.

An interval matrix is defined as
$$
\imace{A}:=\{A\in\R^{m\times n}\mmid \umace{A}\leq A\leq \omace{A}\},
$$
where $ \umace{A}$ and  $\omace{A}$, $\umace{A}\leq\omace{A}$, are given matrices, and the inequality between matrices is understood componentwise. The midpoint and radius matrices are defined as
$$
\Mid{A}:=\frac{1}{2}(\umace{A}+\omace{A}),\quad
\Rad{A}:=\frac{1}{2}(\omace{A}-\umace{A}).
$$
Alternatively, for complex expressions, we write the midpoint and radius matrices as the functions
$$
\MID{(\imace{A})}:=\frac{1}{2}(\umace{A}+\omace{A}),\quad
\RAD{(\imace{A})}:=\frac{1}{2}(\omace{A}-\umace{A}).
$$
The set of all $m\times n$ interval matrices is denoted by $\IR^{m\times n}$.
Naturally, intervals and interval vectors are considered as special cases of interval matrices.
For interval arithmetic see, e.g., \cite{AleHer1983,MooKea2009,Neu1990}.

Given $\imace{A}\in\IR^{m\times n}$ and $\ivr{b}\in\IR^{m}$, the corresponding interval linear system of equations is the family of systems
\begin{align}\label{ilsEq}
Ax=b,\quad A\in\imace{A},\ b\in\ivr{b}.
\end{align}

\paragraph*{Solution concepts.}
There are different definitions of a solution of the interval system \nref{ilsEq}; cf.\ \cite{Fie2006}. We say that $x\in\R^n$ is
\begin{itemize}
\item
a (weak) solution if $\exists A\in\imace{A}$, $\exists b\in\ivr{b}$: $Ax=b$,
\item
a strong solution if $\forall A\in\imace{A}$, $\forall b\in\ivr{b}$: $Ax=b$,
\item
a tolerable solution if $\forall A\in\imace{A}$, $\exists b\in\ivr{b}$: $Ax=b$,
\item
a controllable solution if $\forall b\in\ivr{b}$, $\exists A\in\imace{A}$: $Ax=b$.
\end{itemize}
Similarly, we define analogous solutions for other types of linear systems (inequalities, or mixed equations and inequalities).

Weak solutions are the most commonly used ones  \cite{MooKea2009,Neu1990}. Strong solutions are more appropriate in the context of interval inequalities  \cite{Fie2006,Hla2013b}. Tolerable solutions were studied, e.g., by \cite{PopHla2013,Roh1986,Sha2002,Sha2004}, and controllable solutions in  \cite{PopHla2013,Sha1992,Sha2002}.

The above solution types were generalized to the so called AE solutions \cite{Gol2005,GolChab2006,PopHla2013,Sha2002}. Each interval is associated either with the universal, or with the existential quantifier. Thus, we can split the interval matrix as $\imace{A}=\imace{A}^{\forall}+\imace{A}^{\exists}$, where $\imace{A}^{\forall}$ is the interval matrix comprising universally quantified coefficients, and  $\imace{A}^{\exists}$ concerns existentially quantified coefficients.
Similarly, we decompose the right-hand side vector $\ivr{b}=\ivr{b}^{\forall}+\ivr{b}^{\exists}$. Now, $x\in\R^n$ is an AE solution if 
%\begin{itemize}
%\item
%$\forall A^{\forall}\in\imace{A}^{\forall}$,
%$\forall b^{\forall}\in\ivr{b}^{\forall}$,
%$\exists A^{\exists}\in\imace{A}^{\exists}$,
%$\exists b^{\exists}\in\ivr{b}^{\exists}$: 
%$(A^{\forall}+A^{\exists})x=b^{\forall}+b^{\exists}$.
%\end{itemize}
\begin{align*}
\forall A^{\forall}\in\imace{A}^{\forall},
\forall b^{\forall}\in\ivr{b}^{\forall},
\exists A^{\exists}\in\imace{A}^{\exists},
\exists b^{\exists}\in\ivr{b}^{\exists}:\,  
(A^{\forall}+A^{\exists})x=b^{\forall}+b^{\exists}.
\end{align*}
In the same manner we define AE solutions for interval inequalities and other interval linear systems.

The following characterization of AE solutions is from \cite{Sha2002}.

\begin{theorem}\label{thmAErohn}
A vector $x\in\R^n$ is an AE-solution to interval equations $\imace{A}x=\ivr{b}$ if and only if
\begin{align*}
|\Mid{A}x-\Mid{b}|\leq
%\big(\RAD{(\imace{A}^{\exists})}-\RAD{(\imace{A}^{\forall})}\big)|x|
%+\RAD{(\ivr{b}^{\exists})}-\RAD{(\ivr{b}^{\forall})}.
\big(\Rad{\mace{A}^{\exists}}-\Rad{{A}^{\forall}}\big)|x|
+\Rad{\vr{b}^{\exists}}-\Rad{{b}^{\forall}}.
\end{align*}
\end{theorem}

\paragraph*{Goal.}
The purpose of this paper is to generalize the above characterization of AE solutions to general interval systems, including inequalities or mixed systems of equations and inequalities (Section~\ref{sAeSols}).
The second focus is on the related problem of AE solvability (Section~\ref{sAeSolty}), which is, however, a more difficult problem.

%%%%%%%%%%%%%%%%%%%%%%%%%%%%%%%%%%%%%%%%%%%%%%%%%%%%%%%%%%%%%%% 
\section{AE solutions for the general case}\label{sAeSols}

%%%%%%%%%%%%%%%%%%%%%%%%%%%%%%%%%%%%%%%%%%%%%%%%%%%%%%%%%%%%%%% 
\subsection{Description}

Before we state a characterization of AE solutions for the general case, we state a specific case of inequalities first.

\begin{proposition}\label{proAEineq}
A vector $x\in\R^n$ is an AE-solution to interval inequalities $\imace{A}x\leq\ivr{b}$ if and only if
\begin{align}\label{desAEineq}
\Mid{A}x-\Mid{b}\leq
%\big(\RAD{(\imace{A}^{\exists})}-\RAD{(\imace{A}^{\forall})}\big)|x|
%+\RAD{(\ivr{b}^{\exists})}-\RAD{(\ivr{b}^{\forall})}.
\big(\Rad{\mace{A}^{\exists}}-\Rad{\mace{A}^{\forall}}\big)|x|
+\Rad{\vr{b}^{\exists}}-\Rad{\vr{b}^{\forall}}.
\end{align}
\end{proposition}

\begin{proof}
An AE solution must satisfy
\begin{align*}
 \forall A^{\forall}\in\imace{A}^{\forall}\,
\forall b^{\forall}\in\ivr{b}^{\forall}\,
\exists A^{\exists}\in\imace{A}^{\exists}\,
\exists b^{\exists}\in\ivr{b}^{\exists}:
{A}^{\forall}x-{b}^{\forall}\leq{b}^{\exists}-A^{\exists}x.
\end{align*}
By eliminating the existential  quantifiers, we equivalently get
\begin{align*}
\forall A^{\forall}\in\imace{A}^{\forall}\,
\forall b^{\forall}\in\ivr{b}^{\forall}:
{A}^{\forall}x-{b}^{\forall}
\leq\ovr{\ivr{b}^{\exists}-\imace{A}^{\exists}x},
\end{align*}
and by eliminating the universal quantifiers, we arrive at
\begin{align*}
\ovr{\imace{A}^{\forall}x-\ivr{b}^{\forall}}
\leq\ovr{\ivr{b}^{\exists}-\imace{A}^{\exists}x}.
\end{align*}
This condition can be formulated as
\begin{align*}
%\MID{(\imace{A}^{\forall})}x
%+\RAD{(\imace{A}^{\forall})}|x|-\uvr{b}^{\forall}
%\leq \ovr{b}^{\exists} -\MID{(\imace{A}^{\exists})}x
%+\RAD{(\imace{A}^{\exists})}|x|
\Mid{\mace{A}^{\forall}}x
+\Rad{\mace{A}^{\forall}}|x|-\uvr{b}^{\forall}
\leq \ovr{b}^{\exists} -\Mid{\mace{A}^{\exists}}x
+\Rad{\mace{A}^{\exists}}|x|,
\end{align*}
which is equivalent to the form \nref{desAEineq}.
\end{proof}

Now, we extend the above results to an interval system in a general form. Consider a linear system 
\begin{align}\label{lsGen}
{A}x+{B}y={a},\ 
{C}x+{D}y\leq {b},\ x\geq0,
\end{align}
where the constraint matrices and right-hand side vectors vary in given
interval matrices $\imace{A}\in\IR^{m\times n}$, $\imace{B}\in\IR^{m\times n'}$, $\imace{C}\in\IR^{m'\times n}$, $\imace{D}\in\IR^{m'\times n'}$, and interval vectors $\ivr{a}\in\IR^{m}$, and $\ivr{b}\in\IR^{m'}$. We briefly denote this interval system as
\begin{align}\label{ilsGen}
\imace{A}x+\imace{B}y=\ivr{a},\ 
\imace{C}x+\imace{D}y\leq \ivr{b},\ x\geq0.
\end{align}
Each interval linear system can be transformed to this formulation \cite{Hla2012a}, so it serves as a general form of interval linear systems.

\begin{proposition}\label{proAEgen}
%Let $\imace{A}\in\IR^{m\times n}$, $\imace{B}\in\IR^{m\times n'}$, $\imace{C}\in\IR^{m'\times n}$, $\imace{D}\in\IR^{m'\times n'}$, $\ivr{a}\in\IR^{m}$, and $\ivr{b}\in\IR^{m'}$.
A pair of vectors $(x,y)\in\R^{n+n'}$ is an AE-solution to 
\nref{ilsGen}
%an interval linear system
%\begin{align}
%\imace{A}x+\imace{B}y=\ivr{b},\ 
%\imace{C}x+\imace{D}y\leq \ivr{a},\ x\geq0,
%\end{align}
if and only if
\begin{subequations}\label{desAEgen}
\begin{align}\label{desAEgen1}
|\Mid{A}x+\Mid{B}y-\Mid{a}|&\leq
% \big(\RAD{(\imace{A}^{\exists})}-\RAD{(\imace{A}^{\forall})}\big)x
%+\big(\RAD{(\imace{B}^{\exists})}-\RAD{(\imace{B}^{\forall})}\big)|y|\\
%&\quad+\RAD{(\ivr{a}^{\exists})}-\RAD{(\ivr{a}^{\forall})},\\\label{desAEgen2}
% \Mid{C}x+\Mid{D}y-\Mid{b}&\leq
% \big(\RAD{(\imace{C}^{\exists})}-\RAD{(\imace{C}^{\forall})}\big)x
%+\big(\RAD{(\imace{D}^{\exists})}-\RAD{(\imace{D}^{\forall})}\big)|y|\\
%&\quad+\RAD{(\ivr{b}^{\exists})}-\RAD{(\ivr{b}^{\forall})},\\
 \big(\Rad{\mace{A}^{\exists}}-\Rad{\mace{A}^{\forall}}\big)x
+\big(\Rad{\mace{B}^{\exists}}-\Rad{\mace{B}^{\forall}}\big)|y|
+\Rad{\vr{a}^{\exists}}-\Rad{\vr{a}^{\forall}},\\\label{desAEgen2}
 \Mid{C}x+\Mid{D}y-\Mid{b}&\leq
 \big(\Rad{\mace{C}^{\exists}}-\Rad{\mace{C}^{\forall}}\big)x
+\big(\Rad{\mace{D}^{\exists}}-\Rad{\mace{D}^{\forall}}\big)|y|
+\Rad{\vr{b}^{\exists}}-\Rad{\vr{b}^{\forall}},\\
 x&\geq0.
\end{align}\end{subequations}
\end{proposition}

\begin{proof}
\nref{desAEgen1} follows from Theorem~\ref{thmAErohn} applied on $\imace{A}x+\imace{B}y=\ivr{b}$ and utilizing nonnegativity of $x$. Similarly, \nref{desAEgen2} follows from Proposition~\ref{proAEineq} applied on $\imace{C}x+\imace{D}y\leq \ivr{a}$.
\end{proof}

In Table~\ref{tabAeFeas}, we list some special cases of AE solutions and feasibility.
For equations, we get almost the same results as in \cite{Fie2006}; the only difference is for strong solutions. However, in view of $\uvr{b}\leq\ovr{b}$ the condition $\umace{A}x\geq\ovr{b},\ \omace{A}x\leq\uvr{b},\ x\geq0$ equivalently draws $\umace{A}x=\ovr{b}=\omace{A}x=\uvr{b},\ x\geq0$, or $\Mid{A}x=\Mid{b}$, $\Rad{A}x=\Rad{b}=0$, which is the characterization from \cite{Fie2006}.
For weak and strong solutions of inequalities (both cases), the characterizations also coincide to known results \cite{Fie2006,Hla2013b,RohKre1994}. The other cases (tolerable and controllable solutions of inequalities) seem not to be published yet.

\begin{comment}
\begin{table*}[t]%\small%\footnotesize
% \tabcolsep=3pt
\renewcommand\arraystretch{1.3}
\caption{Robust feasibility conditions for diverse AE solution types and linear programming formulations.\label{tabAeFeas}}
\begin{center}
% \tabcolsep=3pt
\begin{tabular}{@{}cccc@{}}
\toprule
AE solution type & $\imace{A}x=\ivr{b},\ x\geq0$
 & $\imace{A}x\leq \ivr{b}$ & $\imace{A}x\leq \ivr{b},\ x\geq0$\\
\midrule 
weak & $\umace{A}x\leq\ovr{b},\ \omace{A}x\geq\uvr{b},\ x\geq0$
 & $\Mid{A}x\leq\Rad{A}|x|+\ovr{b}$
 & $\umace{A}x\leq\ovr{b},\ x\geq0$\\
strong & $\umace{A}x\geq\ovr{b},\ \omace{A}x\leq\uvr{b},\ x\geq0$
 & $\Mid{A}x+\Rad{A}|x|\leq\uvr{b}$
 & $\omace{A}x\leq\uvr{b},\ x\geq0$\\
tolerable & $\omace{A}x\leq\ovr{b},\ \umace{A}x\geq\uvr{b},\ x\geq0$
 & $\Mid{A}x+\Rad{A}|x|\leq\ovr{b}$
 & $\omace{A}x\leq\uvr{b},\ x\geq0$\\
controllable & $\omace{A}x\leq\uvr{b},\ \umace{A}x\geq\ovr{b},\ x\geq0$
 & $\Mid{A}x-\Rad{A}|x|\leq\uvr{b}$
 & $\umace{A}x\leq\uvr{b},\ x\geq0$\\
\bottomrule
\end{tabular}
\end{center}
\end{table*}
\end{comment}

\begin{table*}[t]%\small%\footnotesize
\renewcommand\arraystretch{1.3}
\caption{Robust feasibility conditions for diverse AE solution types and linear programming formulations.\label{tabAeFeas}}
\begin{center}
\begin{tabular}{@{}ccccc@{}}
\toprule
solution type & $Ax=b$ &$Ax=b,\ x\geq0$ & $Ax\leq b$
 & $Ax\leq b,\ x\geq0$\\
\midrule 
weak &  $|\Mid{A}x-\Mid{b}|\leq\Rad{A}|x|+\Rad{b}$
 & $\umace{A}x\leq\ovr{b},\ \omace{A}x\geq\uvr{b},\ x\geq0$
 & $\Mid{A}x\leq\Rad{A}|x|+\ovr{b}$
 & $\umace{A}x\leq\ovr{b},\ x\geq0$\\
strong & $|\Mid{A}x-\Mid{b}|\leq-\Rad{A}|x|-\Rad{b}$
 & $\umace{A}x\geq\ovr{b},\ \omace{A}x\leq\uvr{b},\ x\geq0$
 & $\Mid{A}x+\Rad{A}|x|\leq\uvr{b}$
 & $\omace{A}x\leq\uvr{b},\ x\geq0$\\
tolerable & $|\Mid{A}x-\Mid{b}|\leq-\Rad{A}|x|+\Rad{b}$
 & $\omace{A}x\leq\ovr{b},\ \umace{A}x\geq\uvr{b},\ x\geq0$
 & $\Mid{A}x+\Rad{A}|x|\leq\ovr{b}$
 & $\omace{A}x\leq\uvr{b},\ x\geq0$\\
controllable & $|\Mid{A}x-\Mid{b}|\leq\Rad{A}|x|-\Rad{b}$
 & $\omace{A}x\leq\uvr{b},\ \umace{A}x\geq\ovr{b},\ x\geq0$
 & $\Mid{A}x-\Rad{A}|x|\leq\uvr{b}$
 & $\umace{A}x\leq\uvr{b},\ x\geq0$\\
\bottomrule
\end{tabular}
\end{center}
\end{table*}

\begin{proposition}
The set of AE solutions is a union of at most $2^{m'}$ convex polyhedral sets.
\end{proposition}

\begin{proof}
Let $s\in\{\pm1\}^{n'}$. The AE solution set in the orthant $\diag(s)y\geq 0$ reads
\begin{align*}
\Mid{A}x+\Mid{B}y-\Mid{a}&\leq
% \big(\RAD{(\imace{A}^{\exists})}-\RAD{(\imace{A}^{\forall})}\big)x
%+\big(\RAD{(\imace{B}^{\exists})}-\RAD{(\imace{B}^{\forall})}\big)\diag(s)y\\
%-\Mid{A}x-\Mid{B}y+\Mid{a}&\leq
% \big(\RAD{(\imace{A}^{\exists})}-\RAD{(\imace{A}^{\forall})}\big)x
%+\big(\RAD{(\imace{B}^{\exists})}-\RAD{(\imace{B}^{\forall})}\big)\diag(s)y\\
%&\quad+\RAD{(\ivr{a}^{\exists})}-\RAD{(\ivr{a}^{\forall})},\\
% \Mid{C}x+\Mid{D}y-\Mid{b}&\leq
% \big(\RAD{(\imace{C}^{\exists})}-\RAD{(\imace{C}^{\forall})}\big)x
%+\big(\RAD{(\imace{D}^{\exists})}-\RAD{(\imace{D}^{\forall})}\big)\diag(s)y\\
%&\quad+\RAD{(\ivr{b}^{\exists})}-\RAD{(\ivr{b}^{\forall})},\\
 \big(\Rad{\mace{A}^{\exists}}-\Rad{\mace{A}^{\forall}}\big)x
+\big(\Rad{\mace{B}^{\exists}}-\Rad{\mace{B}^{\forall}}\big)\diag(s)y
+\Rad{\vr{a}^{\exists}}-\Rad{\vr{a}^{\forall}},\\
-\Mid{A}x-\Mid{B}y+\Mid{a}&\leq
 \big(\Rad{\mace{A}^{\exists}}-\Rad{\mace{A}^{\forall}}\big)x
+\big(\Rad{\mace{B}^{\exists}}-\Rad{\mace{B}^{\forall}}\big)\diag(s)y
+\Rad{\vr{a}^{\exists}}-\Rad{\vr{a}^{\forall}},\\
 \Mid{C}x+\Mid{D}y-\Mid{b}&\leq
 \big(\Rad{\mace{C}^{\exists}}-\Rad{\mace{C}^{\forall}}\big)x
+\big(\Rad{\mace{D}^{\exists}}-\Rad{\mace{D}^{\forall}}\big)\diag(s)y
+\Rad{\vr{b}^{\exists}}-\Rad{\vr{b}^{\forall}},\\
 x&\geq0
\end{align*}
or, equivalently,
\begin{subequations}\label{desPfThmDecomp}
\begin{align}
%(\umace{A}^{\exists}+\omace{A}^{\forall})x
%+\big(\Mid{B}-\RAD{(\imace{B}^{\exists})}\diag(s)
%   +\RAD{(\imace{B}^{\forall})}\diag(s)\big)y&\leq
%\ovr{a}^{\exists}+\uvr{a}^{\forall},\\
%-(\omace{A}^{\exists}+\umace{A}^{\forall})x
%+\big(-\Mid{B}-\RAD{(\imace{B}^{\exists})}\diag(s)
%   +\RAD{(\imace{B}^{\forall})}\diag(s)\big)y&\leq
%-\uvr{a}^{\exists}-\ovr{a}^{\forall},\\
%(\umace{C}^{\exists}+\omace{C}^{\forall})x
%+\big(\Mid{D}-\RAD{(\imace{D}^{\exists})}\diag(s)
%   +\RAD{(\imace{D}^{\forall})}\diag(s)\big)y&\leq
%\ovr{b}^{\exists}+\uvr{b}^{\forall},\\
% x&\geq0
(\umace{A}^{\exists}+\omace{A}^{\forall})x
+\big(\Mid{B}-\Rad{\mace{B}^{\exists}}\diag(s)
   +\Rad{\mace{B}^{\forall}}\diag(s)\big)y&\leq
\ovr{a}^{\exists}+\uvr{a}^{\forall},\\
-(\omace{A}^{\exists}+\umace{A}^{\forall})x
+\big(-\Mid{B}-\Rad{\mace{B}^{\exists}}\diag(s)
   +\Rad{\mace{B}^{\forall}}\diag(s)\big)y&\leq
-\uvr{a}^{\exists}-\ovr{a}^{\forall},\\
(\umace{C}^{\exists}+\omace{C}^{\forall})x
+\big(\Mid{D}-\Rad{\mace{D}^{\exists}}\diag(s)
   +\Rad{\mace{D}^{\forall}}\diag(s)\big)y&\leq
\ovr{b}^{\exists}+\uvr{b}^{\forall},\\
 x&\geq0
\end{align}\end{subequations}
This is a linear system of inequalities, which describes a convex polyhedral set. Therefore, the overall AE solution set is a union of such polyhedral sets subject to to all sign vectors $s\in\{\pm1\}^{n'}$.
\end{proof}

As a consequence, we obtain the following method for finding an AE solution.

\begin{corollary}
A pair of vectors $(x,y)\in\R^{n+n'}$ is an AE-solution to 
\nref{ilsGen} if and inly if they satisfy \nref{desPfThmDecomp} for some $s\in\{\pm1\}^{n'}$.
\end{corollary}

%%%%%%%%%%%%%%%%%%%%%%%%%%%%%%%%%%%%%%%%%%%%%%%%%%%%%%%%%%%%%%% 
\subsection{Attainment}

Given an AE solution $(x^*,y^*)$, the following natural question arises: For a realization of $\forall$-parameters, what are the values of $\exists$-parameters, for which $(x^*,y^*)$ remains to be a solution?

\begin{proposition}
Let $A^{\forall}\in\imace{A}^{\forall}$, $B^{\forall}\in\imace{B}^{\forall}$, $C^{\forall}\in\imace{C}^{\forall}$, $D^{\forall}\in\imace{D}^{\forall}$, $a^{\forall}\in\ivr{a}^{\forall}$, and $b^{\forall}\in\ivr{b}^{\forall}$. Then $(x^*,y^*)$ solves \nref{lsGen} for the setting
\begin{subequations}\label{eqPropAttain}
\begin{align}
%A^{\exists}&=\MID{(\imace{A}^{\exists})}
%  -\diag{(u)}\RAD{(\imace{A}^{\exists})},\\
%B^{\exists}&=\MID{(\imace{B}^{\exists})}
% - \diag{(u)}\RAD{(\imace{B}^{\exists})}\diag{(\sgn(y^*))},\\ 
%a^{\exists}&=\MID{(\ivr{a}^{\exists})}
% +\diag{(u)}\RAD{(\ivr{b}^{\exists})},\\ 
%C^{\exists}&=\umace{C}^{\exists},\\ 
%D^{\exists}&=\MID{(\imace{D}^{\exists})}
% -\RAD{(\imace{D}^{\exists})}\diag{(\sgn(y))},\\
%b^{\exists}&=\ovr{b}^{\exists},
A^{\exists}&=\Mid{\mace{A}^{\exists}}
  -\diag{(u)}\Rad{\mace{A}^{\exists}},\\
B^{\exists}&=\Mid{\mace{B}^{\exists}}
 - \diag{(u)}\Rad{\mace{B}^{\exists}}\diag{(\sgn(y^*))},\\ 
a^{\exists}&=\Mid{\vr{a}^{\exists}}
 +\diag{(u)}\Rad{\vr{b}^{\exists}},\\ 
C^{\exists}&=\umace{C}^{\exists},\\ 
D^{\exists}&=\Mid{\mace{D}^{\exists}}
 -\Rad{\mace{D}^{\exists}}\diag{(\sgn(y))},\\
b^{\exists}&=\ovr{b}^{\exists},
\end{align}\end{subequations}
where $u\in[-1,1]^m$ is defined entrywise as
\begin{align*}
u_i=\begin{cases}\displaystyle
%\frac{(\Mid{A}x+\Mid{B}y-\Mid{b})_i}
%{\big(\RAD{(\imace{A}^{\exists})}x
% +\RAD{(\imace{B}^{\exists})}|y|+\RAD{(\ivr{b}^{\exists})}\big)_i} 
%& \mbox{if }
%  \big(\RAD{(\imace{A}^{\exists})}x
%  +\RAD{(\imace{B}^{\exists})}|y|+\RAD{(\ivr{b}^{\exists})}\big)_i>0,\\
%1&\mbox{otherwise}.
\frac{(\Mid{A}x+\Mid{B}y-\Mid{a})_i}
{\big(\Rad{\mace{A}^{\exists}}x
 +\Rad{\mace{B}^{\exists}}|y|+\Rad{\vr{a}^{\exists}}\big)_i} 
& \mbox{if }
  \big(\Rad{\mace{A}^{\exists}}x
  +\Rad{\mace{B}^{\exists}}|y|+\Rad{\vr{a}^{\exists}}\big)_i>0,\\
1&\mbox{otherwise}.
\end{cases}
%\qquad i=1,\dots,m.
\end{align*}
\end{proposition}

\begin{proof}
Since $(x^*,y^*)$ is an AE solution, it must be a weak solution to 
\begin{align*}
(\imace{A}^{\exists}+A^{\forall})x
 +(\imace{B}^{\exists}+B^{\forall})y
=\ivr{a}^{\exists}+a^{\forall},\ 
(\imace{C}^{\exists}+C^{\forall})x
 +(\imace{D}^{\exists}+D^{\forall})y
\leq \ivr{b}^{\exists}+b^{\forall},\ x\geq0.
\end{align*}
By Hlad\'{i}k \cite{Hla2013b}, $(x^*,y^*)$ solves the constraints for \nref{eqPropAttain}.
\end{proof}

%%%%%%%%%%%%%%%%%%%%%%%%%%%%%%%%%%%%%%%%%%%%%%%%%%%%%%%%%%%%%%% 
\subsection{Complexity}

Checking whether a given pair $(x,y)$ is an AE solution is easy by checking  \nref{desAEgen}. On the other hand, computing an AE solution, or just checking whether there exists any AE solution, may be a computationally hard problem.

Some special cases of \nref{desAEgen} are polynomially solvable by reducing to linear system of equations and inequalities and utilizing polynomiality of linear programming \cite{Schr1998}.

\begin{proposition}
%If $\RAD{\imace{B}^{\exists}}=0$ and $\RAD{\imace{D}^{\exists}}=0$, then computing an AE solution or checking its existence is a polynomial time problem.
If $\Rad{\mace{B}^{\exists}}=0$ and $\Rad{\mace{D}^{\exists}}=0$, then computing an AE solution or checking its existence is a polynomial time problem.
\end{proposition}

\begin{proof}
Under the assumption, the AE solution set is described by
\begin{align*}
|\Mid{A}x+\Mid{B}y-\Mid{a}|&\leq
% \big(\RAD{(\imace{A}^{\exists})}-\RAD{(\imace{A}^{\forall})}\big)x
%-\RAD{(\imace{B}^{\forall})}|y|
%+\RAD{(\ivr{a}^{\exists})}-\RAD{(\ivr{a}^{\forall})},\\
% \Mid{C}x+\Mid{D}y-\Mid{b}&\leq
% \big(\RAD{(\imace{C}^{\exists})}-\RAD{(\imace{C}^{\forall})}\big)x
%-\RAD{(\imace{D}^{\forall})}|y|
%+\RAD{(\ivr{b}^{\exists})}-\RAD{(\ivr{b}^{\forall})},\quad
 \big(\Rad{\mace{A}^{\exists}}-\Rad{\mace{A}^{\forall}}\big)x
-\Rad{\mace{B}^{\forall}}|y|
+\Rad{\vr{a}^{\exists}}-\Rad{\vr{a}^{\forall}},\\
 \Mid{C}x+\Mid{D}y-\Mid{b}&\leq
 \big(\Rad{\mace{C}^{\exists}}-\Rad{\mace{C}^{\forall}}\big)x
-\Rad{\mace{D}^{\forall}}|y|
+\Rad{\vr{b}^{\exists}}-\Rad{\vr{b}^{\forall}},\quad
 x\geq0.
\end{align*}
Equivalently, it is the projection to the $(x,y)$-subspace of the convex polyhedral set described in $(x,y,z)$-space as
\begin{align*}
\Mid{A}x+\Mid{B}y-\Mid{a}
%+\RAD{(\imace{B}^{\forall})}z&\leq
% \big(\RAD{(\imace{A}^{\exists})}-\RAD{(\imace{A}^{\forall})}\big)x
%+\RAD{(\ivr{a}^{\exists})}-\RAD{(\ivr{a}^{\forall})},\\
%-\Mid{A}x-\Mid{B}y+\Mid{a}
%+\RAD{(\imace{B}^{\forall})}z&\leq
% \big(\RAD{(\imace{A}^{\exists})}-\RAD{(\imace{A}^{\forall})}\big)x
%+\RAD{(\ivr{a}^{\exists})}-\RAD{(\ivr{a}^{\forall})},\\
% \Mid{C}x+\Mid{D}y-\Mid{b}
%+\RAD{(\imace{D}^{\forall})}z&\leq
% \big(\RAD{(\imace{C}^{\exists})}-\RAD{(\imace{C}^{\forall})}\big)x
%+\RAD{(\ivr{b}^{\exists})}-\RAD{(\ivr{b}^{\forall})},\\\
+\Rad{\mace{B}^{\forall}}z&\leq
 \big(\Rad{\mace{A}^{\exists}}-\Rad{\mace{A}^{\forall}}\big)x
+\Rad{\vr{a}^{\exists}}-\Rad{\vr{a}^{\forall}},\\
-\Mid{A}x-\Mid{B}y+\Mid{a}
+\Rad{\mace{B}^{\forall}}z&\leq
 \big(\Rad{\mace{A}^{\exists}}-\Rad{\mace{A}^{\forall}}\big)x
+\Rad{\vr{a}^{\exists}}-\Rad{\vr{a}^{\forall}},\\
 \Mid{C}x+\Mid{D}y-\Mid{b}
+\Rad{\mace{D}^{\forall}}z&\leq
 \big(\Rad{\mace{C}^{\exists}}-\Rad{\mace{C}^{\forall}}\big)x
+\Rad{\vr{b}^{\exists}}-\Rad{\vr{b}^{\forall}},\\\
 x&\geq0,\quad y\leq z,\quad -y\leq z,
\end{align*}
or, in an alternative form
\begin{subequations}\label{desPfThmPol}
\begin{align}
(\umace{A}^{\exists}+\omace{A}^{\forall})x
+\Mid{B}y
%+\RAD{(\imace{B}^{\forall})}z&\leq
%\ovr{a}^{\exists}+\uvr{a}^{\forall},\\
%-(\omace{A}^{\exists}+\umace{A}^{\forall})x
%-\Mid{B}y
%+\RAD{(\imace{B}^{\forall})}z&\leq
%-\uvr{a}^{\exists}-\ovr{a}^{\forall},\\
%(\umace{C}^{\exists}+\omace{C}^{\forall})x
%+\Mid{D}y
%+\RAD{(\imace{D}^{\forall})}z&\leq
%\ovr{b}^{\exists}+\uvr{b}^{\forall},\\
+\Rad{\mace{B}^{\forall}}z&\leq
\ovr{a}^{\exists}+\uvr{a}^{\forall},\\
-(\omace{A}^{\exists}+\umace{A}^{\forall})x
-\Mid{B}y
+\Rad{\mace{B}^{\forall}}z&\leq
-\uvr{a}^{\exists}-\ovr{a}^{\forall},\\
(\umace{C}^{\exists}+\omace{C}^{\forall})x
+\Mid{D}y
+\Rad{\mace{D}^{\forall}}z&\leq
\ovr{b}^{\exists}+\uvr{b}^{\forall},\\
 x&\geq0,\quad y\leq z,\quad -y\leq z.
%\qedhere
\end{align}\end{subequations}
\end{proof}

In general, however, the problem of finding an AE solution is NP-hard. It remains NP-hard even on the following sub-cases:
\begin{itemize}
\item
weak solutions to $\imace{A}x=\ivr{b}$; see \cite{Fie2006,KreLak1998,LakNos1994} 
\item
weak solutions to $\imace{A}x\leq\ivr{b}$; see \cite{Fie2006} 
\item
controllable solutions to $\imace{A}x=\ivr{b}$; see \cite{Fie2006,LakNos1994} 
\end{itemize}

%%%%%%%%%%%%%%%%%%%%%%%%%%%%%%%%%%%%%%%%%%%%%%%%%%%%%%%%%%%%%%% 
\section{AE solvability}\label{sAeSolty}

The interval systems \nref{ilsGen} is called \emph{AE solvable} if for each realization of $\forall$-parameters there are realizations of $\exists$-parameters such that \nref{lsGen} has a solution. Formally,  \nref{ilsGen} is  AE solvable if 
\begin{align*}
&\forall A^{\forall}\in\imace{A}^{\forall},
\forall B^{\forall}\in\imace{B}^{\forall},
\forall C^{\forall}\in\imace{C}^{\forall},
\forall D^{\forall}\in\imace{D}^{\forall},
\forall a^{\forall}\in\ivr{a}^{\forall},
\forall b^{\forall}\in\ivr{b}^{\forall},\\
&\exists A^{\exists}\in\imace{A}^{\exists},
\exists B^{\exists}\in\imace{B}^{\exists},
\exists C^{\exists}\in\imace{C}^{\exists},
\exists D^{\exists}\in\imace{D}^{\exists},
\exists a^{\exists}\in\ivr{a}^{\exists},
\exists b^{\exists}\in\ivr{b}^{\exists}:\,  
(A^{\forall}+A^{\exists})x=b^{\forall}+b^{\exists}
\end{align*}
has a solution.

%%%%%%%%%%%%%%%%%%%%%%%%%%%%%%%%%%%%%%%%%%%%%%%%%%%%%%%%%%%%%%% 
\subsection{When AE solvability meets AE solution existence}\label{ssAeSoltySols}

Notice that as long as \nref{ilsGen} has an AE solution, then it is AE solvable, but the converse implication does not hold in general. For example, the interval system of equations
$$
x_1+x_2=[1,2]
$$
is strongly solvable (solvable for each realization), but there is no strong solution common to all realizations of the interval.

Surprisingly, for strong solvability interval inequalities  $\imace{A}x\leq \ivr{b},$ or $\imace{A}x\leq \ivr{b}$, $x\geq0$ we have equivalence. By \cite{Fie2006,RohKre1994}, there is a strong solution if and only if the interval inequality system is strongly solvable. In Hlad\'{i}k \cite{Hla2013b}, an extended version combining both nonnegative and free variables was presented.

\begin{theorem}[\cite{Hla2013b}]\label{thmStrSolv}
The interval system $\imace{A}x+\imace{B}y\leq \ivr{b}$, $x\geq0$ is strongly solvable if and only if 
$$\omace{A}x+\omace{B}y^1-\umace{B}y^2\leq \uvr{b},\ \ x,y^1,y^2\geq0$$
is solvable.
\end{theorem}

For AE solvability of interval inequalities, we have the following generalization.

\begin{proposition}
For the interval system
\begin{align}\label{ineqAeSolvSolv}
\imace{A}x+\imace{B}^{\forall}y\leq \ivr{b},\ x\geq0
\end{align}
the following are equivalent
\begin{enumerate}[(i)]
\item
$x,y$ is an AE solution of \nref{ineqAeSolvSolv},
\item
$x,y$ solves
\begin{align}\label{desAEineqNonnegAbs}
(\umace{A}^{\exists}+\omace{A}^{\forall})x
%+\MID(\imace{B}^{\forall})y+\RAD(\imace{B}^{\forall})|y|
+\Mid{\mace{B}^{\forall}}y+\Rad{\mace{B}^{\forall}}|y|
\leq \ovr{b}^{\exists}+\uvr{b}^{\forall},\ \ 
 x\geq0,
\end{align}
\item
$x,y$ solves
\begin{align}\label{desAEineqNonneg}
y=y^1-y^2,\ \ 
(\umace{A}^{\exists}+\omace{A}^{\forall})x
+\omace{B}^{\forall}y^1-\umace{B}^{\forall}y^2
\leq \ovr{b}^{\exists}+\uvr{b}^{\forall},\ \ 
 x,y^1,y^2\geq0.
\end{align}
\end{enumerate}
\end{proposition}

\begin{proof}
The equivalence \quo{$(i)\Leftrightarrow(ii)$} follows from Proposition~\ref{proAEgen}.

\quo{$(ii)\Rightarrow(iii)$}
Let $x,y$ be a solution of \nref{desAEineqNonnegAbs}. Put $y^1:=\max(y,0)$ the positive part and $y^2:=\max(-y,0)$ the negative part of $y$. Then $y=y^1-y^2$, $|y|=y^1+y^2$, and \nref{desAEineqNonnegAbs} takes the form of 
\begin{align*}
(\umace{A}^{\exists}+\omace{A}^{\forall})x
%+\MID(\imace{B}^{\forall})(y^1-y^2)+\RAD(\imace{B}^{\forall})(y^1+y^2)
+\Mid{\mace{B}^{\forall}}(y^1-y^2)+\Rad{\mace{B}^{\forall}}(y^1+y^2)
\leq \ovr{b}^{\exists}+\uvr{b}^{\forall},\quad
 x,y^1,y^2\geq0,
\end{align*}
which is equivalent to \nref{desAEineqNonneg}.

\quo{$(ii)\Leftarrow(iii)$}
Let $x,y^1,y^2$ be a solution of \nref{desAEineqNonneg}, and put $y:=y^1-y^2$. Then 
\begin{align*}
%(\umace{A}^{\exists}+\omace{A}^{\forall})x
%+\MID(\imace{B}^{\forall})y+\RAD(\imace{B}^{\forall})|y|
%&=(\umace{A}^{\exists}+\omace{A}^{\forall})x
%+\MID(\imace{B}^{\forall})(y^1-y^2)+\RAD(\imace{B}^{\forall})|y^1-y^2|\\
%&\leq(\umace{A}^{\exists}+\omace{A}^{\forall})x
%+\MID(\imace{B}^{\forall})(y^1-y^2)+\RAD(\imace{B}^{\forall})(y^1+y^2)\\
(\umace{A}^{\exists}+\omace{A}^{\forall})x
+\Mid{\mace{B}^{\forall}}y+\Rad{\mace{B}^{\forall}}|y|
&=(\umace{A}^{\exists}+\omace{A}^{\forall})x
+\Mid{\mace{B}^{\forall}}(y^1-y^2)+\Rad{\mace{B}^{\forall}}|y^1-y^2|\\
&\leq(\umace{A}^{\exists}+\omace{A}^{\forall})x
+\Mid{\mace{B}^{\forall}}(y^1-y^2)+\Rad{\mace{B}^{\forall}}(y^1+y^2)\\
&\leq \ovr{b}^{\exists}+\uvr{b}^{\forall},\ \ 
 x\geq0,
\end{align*}
meaning that $x,y$ solves \nref{desAEineqNonnegAbs}.
\end{proof}

\begin{proposition}
The interval system \nref{ineqAeSolvSolv}
is AE solvable if and only if it has an AE solution.
\end{proposition}

\begin{proof}
First we show that \nref{ineqAeSolvSolv} is AE solvable if and only if
\begin{align}\label{pfIneqAeSolvSolvStr}
(\umace{A}^{\exists}+\imace{A}^{\forall})x+\imace{B}^{\forall}y
\leq \ovr{b}^{\exists}+\ivr{b}^{\forall},\ \ x\geq0
\end{align}
is strongly solvable. If \nref{pfIneqAeSolvSolvStr} is strongly solvable, then for each ${A}^{\forall}\in\imace{A}^{\forall}$, ${B}^{\forall}\in\imace{B}^{\forall}$ and ${b}^{\forall}\in\ivr{b}^{\forall}$, and for the choice ${A}^{\exists}:=\umace{A}^{\exists}$ and ${b}^{\exists}:=\ovr{b}^{\exists}$, the system
\begin{align}\label{pfIneqAeSolvSolvInst}
({A}^{\exists}+{A}^{\forall})x+{B}^{\forall}y
\leq {b}^{\exists}+{b}^{\forall},\ \ x\geq0
\end{align}
has a solution. Contrary, if \nref{ineqAeSolvSolv} is AE solvable, then for each  ${A}^{\forall}\in\imace{A}^{\forall}$, ${B}^{\forall}\in\imace{B}^{\forall}$ and ${b}^{\forall}\in\ivr{b}^{\forall}$, there are ${A}^{\exists}\in\imace{A}^{\exists}$ and ${b}^{\exists}\in\ivr{b}^{\exists}$ such that the system \nref{pfIneqAeSolvSolvInst} is solvable. This implies that
\begin{align*}
(\umace{A}^{\exists}+{A}^{\forall})x+{B}^{\forall}y
\leq \ovr{b}^{\exists}+{b}^{\forall},\ x\geq0
\end{align*}
is solvable, too.

By Theorem~\ref{thmStrSolv}, the interval system \nref{pfIneqAeSolvSolvStr} is strongly solvable if and only if \nref{desAEineqNonneg} holds, that is, if and only if \nref{ineqAeSolvSolv} has an AE solution.
\end{proof}

Notice that the result from \cite{Fie2006,RohKre1994} cannot by generalized to any interval inequality system. Below, we give a counterexample.

\begin{example}\label{exFail}
Consider the interval system of inequalities
\begin{align*}
{A}^{\exists}x\leq -2,\ \ 
{A}^{\forall}x\leq 1,
\end{align*}
where ${A}^{\exists},{A}^{\forall}\in[-1,1]$. This interval system has no AE solution since the system \nref{desAEineq}, which takes the form of
\begin{align*}
2\leq|x|,\ \ |x|\leq1,
\end{align*}
has no solution.

In contrast, the interval system is AE solvable. If ${A}^{\forall}\geq0$, then we can take ${A}^{\exists}:=1$ and $x:=-2$. If ${A}^{\forall}\leq0$, then we can take ${A}^{\exists}:=-1$ and $x:=2$. In summary, the interval system of inequalities is AE solvable, but has no AE solution.
\end{example}

%%%%%%%%%%%%%%%%%%%%%%%%%%%%%%%%%%%%%%%%%%%%%%%%%%%%%%%%%%%%%%% 
\subsection{Conditions for AE solvability}

Let us recall the characterization of weak solutions for a general system of interval equations and inequalities from Hlad\'{i}k \cite{Hla2013b}.

\begin{theorem}\label{thmHladik}
Let $\imace{A}\in\IR^{m\times n}$, $\imace{B}\in\IR^{m\times n'}$, $\imace{C}\in\IR^{m'\times n}$, $\imace{D}\in\IR^{m'\times n'}$, $\ivr{a}\in\IR^{m}$, and $\ivr{b}\in\IR^{m'}$. 
A pair $(x,y)$, $x\in\R^{m}$, $y\in\R^n$, is a weak solution to the interval system
\begin{align*}
\imace{A}x+\imace{B}y=\ivr{a},\ 
\imace{C}x+\imace{D}y\leq \ivr{b},\ x\geq0.
\end{align*}
if and only if there is $s\in\{\pm1\}^n$ such that
\begin{align*}
\umace{A}x+(\Mid{B}-\Rad{B}\diag(s))y&\leq\ovr{b},\\ 
-\omace{A}x-(\Mid{B}+\Rad{B}\diag(s))y&\leq-\uvr{b},\\ 
\umace{C}x+(\Mid{D}-\Rad{D}\diag(s))y&\leq\ovr{d},\ x\geq0.
\end{align*}
\end{theorem}

We will also employ the well known Farkas lemma. In particular, we utilize the following form from Hlad\'{i}k \cite{Hla2013b}.

\begin{lemma}\label{lmmFarkas}
Exactly one of the linear systems
\begin{align*}
Ax+By=b,\ 
Cx+Dy\leq d,\ x\geq0
\end{align*}
and
\begin{align*}
A^Tp+C^Tq\geq0,\ 
B^Tp+D^Tq=0,\ 
b^Tp+d^Tq\leq-1,\ q\geq0,
\end{align*}
is solvable.
\end{lemma}

As long as $\Rad{\mace{B}^{\exists}}=0$ and $\Rad{\mace{D}^{\exists}}=0$, we have a sufficient and necessary characterization of AE solvability for the general model \nref{ilsGen}.

\begin{proposition}\label{propAEsolty}
Suppose that $\Rad{\mace{B}^{\exists}}=0$ and $\Rad{\mace{D}^{\exists}}=0$.
Then \nref{ilsGen} is AE solvable if and only if for each $s\in\{\pm1\}^m$ the system
\begin{subequations}\label{desAEsolty}
\begin{align}
(\umace{A}^{\exists}+\Mid{A}^{\forall}+\diag(s)\Rad{A}^{\forall})x
+(\Mid{B}^{\forall}+\diag(s)\Rad{B}^{\forall})y^1&\\
-(\Mid{B}^{\forall}-\diag(s)\Rad{B}^{\forall})y^2
&\leq\ovr{a}^{\exists}+\Mid{a}^{\forall}-\diag(s)\Rad{a}^{\forall},\\
-(\omace{A}^{\exists}+\Mid{A}^{\forall}+\diag(s)\Rad{A}^{\forall})^Tx
-(\Mid{B}^{\forall}+\diag(s)\Rad{B}^{\forall})y^1&\\
+(\Mid{B}^{\forall}-\diag(s)\Rad{B}^{\forall})y^2
&\leq-\uvr{a}^{\exists}-\Mid{a}^{\forall}+\diag(s)\Rad{a}^{\forall},\\
(\omace{C}^{\forall}+\umace{C}^{\exists})x
+\omace{D}^{\forall}y^1-\umace{D}^{\forall}y^2
&\leq \uvr{b}^{\forall}+\ovr{b}^{\exists},\\
x,y^1,y^2&\geq0
\end{align}\end{subequations}
is solvable.
\end{proposition}

\begin{proof}
The interval system \nref{ilsGen} is not AE solvable if and only if there are $A^{\forall}\in\imace{A}^{\forall}$, $B^{\forall}\in\imace{B}^{\forall}$, $C^{\forall}\in\imace{C}^{\forall}$, $ D^{\forall}\in\imace{D}^{\forall}$, $a^{\forall}\in\ivr{a}^{\forall}$ and $b^{\forall}\in\ivr{b}^{\forall}$ such that the interval system
\begin{align*}
(A^{\forall}+\imace{A}^{\exists})x+B^{\forall}y
 = a^{\forall}+\ivr{a}^{\exists},\ \ 
(C^{\forall}+\imace{C}^{\exists})x+D^{\forall}y
 \leq b^{\forall}+\ivr{b}^{\exists},\ \ 
x\geq0
\end{align*}
is not weakly solvable.
By Theorem~\ref{thmHladik}, equivalently, the real system
\begin{align*}
(A^{\forall}+\umace{A}^{\exists})x+B^{\forall}y
 &\leq a^{\forall}+\ovr{a}^{\exists},\\ 
-(A^{\forall}+\omace{A}^{\exists})x-B^{\forall}y
 &\leq -a^{\forall}-\uvr{a}^{\exists},\\ 
(C^{\forall}+\umace{C}^{\exists})x+D^{\forall}y
 &\leq b^{\forall}+\ovr{b}^{\exists},\\ 
x\geq0
\end{align*}
is not solvable. By the Farkas Lemma~\ref{lmmFarkas}, this is true if and only if the system
\begin{align*}
(A^{\forall}+\umace{A}^{\exists})^Tu
-(A^{\forall}+\omace{A}^{\exists})^Tv
+(C^{\forall}+\umace{C}^{\exists})^Tw&\geq0,\\
(B^{\forall})^Tu-(B^{\forall})^Tv+(D^{\forall})^Tw&=0,\\
(a^{\forall}+\ovr{a}^{\exists})^Tu
-(a^{\forall}+\uvr{a}^{\exists})^Tv
+(b^{\forall}+\ovr{b}^{\exists})^Tw&\leq-1,\\
u,v,w&\geq0
\end{align*}
is solvable. We rewrite the system as
\begin{align*}
(\umace{A}^{\exists})^Tu-(\omace{A}^{\exists})^Tv
+(A^{\forall})^T(u-v)
+(C^{\forall}+\umace{C}^{\exists})^Tw&\geq0,\\
(B^{\forall})^T(u-v)+(D^{\forall})^Tw&=0,\\
(\ovr{a}^{\exists})^Tu-(\uvr{a}^{\exists})^Tv
+(a^{\forall})^T(u-v)
+(b^{\forall}+\ovr{b}^{\exists})^Tw&\leq-1,\\
u,v,w&\geq0.
\end{align*}
Since this system is solvable for some realization of $\forall$-parameters, we have equivalently that the interval system
\begin{align*}
(\umace{A}^{\exists})^Tu-(\omace{A}^{\exists})^Tv
+(\imace{A}^{\forall})^T(u-v)
+(\imace{C}^{\forall}+\umace{C}^{\exists})^Tw&\geq0,\\
(\imace{B}^{\forall})^T(u-v)+(\imace{D}^{\forall})^Tw&=0,\\
(\ovr{a}^{\exists})^Tu-(\uvr{a}^{\exists})^Tv
+(\ivr{a}^{\forall})^T(u-v)
+(\ivr{b}^{\forall}+\ovr{b}^{\exists})^Tw&\leq-1,\\
u,v,w&\geq0
\end{align*}
is weakly solvable. By Theorem~\ref{thmHladik}, there is equivalently $s\in\{\pm1\}^m$ such that
\begin{align*}
(\umace{A}^{\exists})^Tu-(\omace{A}^{\exists})^Tv
+(\Mid{A}^{\forall}+\diag(s)\Rad{A}^{\forall})^T(u-v)
+(\omace{C}^{\forall}+\umace{C}^{\exists})^Tw&\geq0,\\
(\Mid{B}^{\forall}+\Rad{B}^{\forall}\diag(s))^T(u-v)
+\omace{D}^{\forall}w&\geq0,\\
-(\Mid{B}^{\forall}-\diag(s)\Rad{B}^{\forall})^T(u-v)
-(\umace{D}^{\forall})^Tw&\geq0,\\
(\ovr{a}^{\exists})^Tu-(\uvr{a}^{\exists})^Tv
+(\Mid{a}^{\forall}-\diag(s)\Rad{a}^{\forall})^T(u-v)
+(\uvr{b}^{\forall}+\ovr{b}^{\exists})^Tw&\leq-1,\\
u,v,w&\geq0.
\end{align*}
is solvable. By the Farkas lemma again, this system is solvable if and only if the system
\begin{align*}
(\umace{A}^{\exists}+\Mid{A}^{\forall}+\diag(s)\Rad{A}^{\forall})x
+(\Mid{B}^{\forall}+\diag(s)\Rad{B}^{\forall})y^1&\\
-(\Mid{B}^{\forall}-\diag(s)\Rad{B}^{\forall})y^2
&\leq\ovr{a}^{\exists}+\Mid{a}^{\forall}-\diag(s)\Rad{a}^{\forall},\\
-(\omace{A}^{\exists}+\Mid{A}^{\forall}+\diag(s)\Rad{A}^{\forall})^Tx
-(\Mid{B}^{\forall}+\diag(s)\Rad{B}^{\forall})y^1&\\
+(\Mid{B}^{\forall}-\diag(s)\Rad{B}^{\forall})y^2
&\leq-\uvr{a}^{\exists}-\Mid{a}^{\forall}+\diag(s)\Rad{a}^{\forall},\\
(\omace{C}^{\forall}+\umace{C}^{\exists})x
+\omace{D}^{\forall}y^1-\umace{D}^{\forall}y^2
&\leq \uvr{b}^{\forall}+\ovr{b}^{\exists}\\
x,y^1,y^2&\geq0
\end{align*}
is not solvable, which completes the proof.
\end{proof}

For the general case, the idea from the proof fails. However, a small adaptation of the proof gives a necessary condition and a sufficient condition for AE solvability. Notice that they differ only in the order of quantifiers.

\begin{proposition}\label{propAeSoltySuff}
The interval system \nref{ilsGen} is AE solvable if there is $z\in\{\pm1\}^{n'}$ such that for each $s\in\{\pm1\}^m$ the system
\begin{subequations}\label{desAEsoltySuf}
\begin{align}
(\umace{A}^{\exists}+\Mid{A}^{\forall}+\diag(s)\Rad{A}^{\forall})x&\\
+(\Mid{B}^{\exists}-\Rad{B}^{\exists}\diag(z)
  +\Mid{B}^{\forall}+\diag(s)\Rad{B}^{\forall})y^1&\\
-(\Mid{B}^{\exists}-\Rad{B}^{\exists}\diag(z)
  +\Mid{B}^{\forall}-\diag(s)\Rad{B}^{\forall})y^2
&\leq\ovr{a}^{\exists}+\Mid{a}^{\forall}-\diag(s)\Rad{a}^{\forall},\\
-(\omace{A}^{\exists}+\Mid{A}^{\forall}+\diag(s)\Rad{A}^{\forall})^Tx&\\
-(\Mid{B}^{\exists}+\Rad{B}^{\exists}\diag(z)
  +\Mid{B}^{\forall}+\diag(s)\Rad{B}^{\forall})y^1&\\
+(\Mid{B}^{\exists}+\Rad{B}^{\exists}\diag(z)
  +\Mid{B}^{\forall}-\diag(s)\Rad{B}^{\forall})y^2
&\leq-\uvr{a}^{\exists}-\Mid{a}^{\forall}+\diag(s)\Rad{a}^{\forall},\\
(\omace{C}^{\forall}+\umace{C}^{\exists})x
+(\omace{D}^{\forall}+\Mid{D}^{\exists}-\Rad{D}^{\exists}\diag(z))y^1&\\
-(\umace{D}^{\forall}+\Mid{D}^{\exists}-\Rad{D}^{\exists}\diag(z))y^2
&\leq \uvr{b}^{\forall}+\ovr{b}^{\exists}\\
x,y^1,y^2&\geq0
\end{align}\end{subequations}
is solvable.
\end{proposition}

\begin{proof}
If the interval system \nref{ilsGen} is not AE solvable, then there are $A^{\forall}\in\imace{A}^{\forall}$, $B^{\forall}\in\imace{B}^{\forall}$, $C^{\forall}\in\imace{C}^{\forall}$, $ D^{\forall}\in\imace{D}^{\forall}$, $a^{\forall}\in\ivr{a}^{\forall}$ and $b^{\forall}\in\ivr{b}^{\forall}$ such that the interval system
\begin{align*}
(A^{\forall}+\imace{A}^{\exists})x+(B^{\forall}+\imace{B}^{\exists})y
 = a^{\forall}+\ivr{a}^{\exists},\ \ 
(C^{\forall}+\imace{C}^{\exists})x+(D^{\forall}+\imace{D}^{\exists})y
 \leq b^{\forall}+\ivr{b}^{\exists},\ \ 
x\geq0
\end{align*}
is not weakly solvable.
By Theorem~\ref{thmHladik}, for each $z\in\{\pm1\}^{n'}$ the real system
\begin{align*}
(A^{\forall}+\umace{A}^{\exists})x
+(B^{\forall}+\Mid{B}^{\exists}-\Rad{B}^{\exists}\diag(z))y
 &\leq a^{\forall}+\ovr{a}^{\exists},\\ 
-(A^{\forall}+\omace{A}^{\exists})x
-(B^{\forall}+\Mid{B}^{\exists}+\Rad{B}^{\exists}\diag(z))y
 &\leq -a^{\forall}-\uvr{a}^{\exists},\\ 
(C^{\forall}+\umace{C}^{\exists})x
+(D^{\forall}+\Mid{D}^{\exists}-\Rad{D}^{\exists}\diag(z))y
 &\leq b^{\forall}+\ovr{b}^{\exists},\\ 
x\geq0
\end{align*}
is not solvable. By the Farkas Lemma~\ref{lmmFarkas}, this is true if and only if the system
\begin{align*}
(A^{\forall}+\umace{A}^{\exists})^Tu
-(A^{\forall}+\omace{A}^{\exists})^Tv
+(C^{\forall}+\umace{C}^{\exists})^Tw&\geq0,\\
(B^{\forall}+\Mid{B}^{\exists}-\Rad{B}^{\exists}\diag(z))^Tu
-(B^{\forall}+\Mid{B}^{\exists}+\Rad{B}^{\exists}\diag(z))^Tv&\\
+(D^{\forall}+\Mid{D}^{\exists}-\Rad{D}^{\exists}\diag(z))^Tw&=0,\\
(a^{\forall}+\ovr{a}^{\exists})^Tu
-(a^{\forall}+\uvr{a}^{\exists})^Tv
+(b^{\forall}+\ovr{b}^{\exists})^Tw&\leq-1,\\
u,v,w&\geq0
\end{align*}
is solvable. We rewrite the system as
\begin{align*}
(\umace{A}^{\exists})^Tu-(\omace{A}^{\exists})^Tv
+(A^{\forall})^T(u-v)
+(C^{\forall}+\umace{C}^{\exists})^Tw&\geq0,\\
(\Mid{B}^{\exists}-\Rad{B}^{\exists}\diag(z))^Tu
-(\Mid{B}^{\exists}+\Rad{B}^{\exists}\diag(z))^Tv&\\
+(B^{\forall})^T(u-v)
+(D^{\forall}+\Mid{D}^{\exists}-\Rad{D}^{\exists}\diag(z))^Tw&=0,\\
(\ovr{a}^{\exists})^Tu-(\uvr{a}^{\exists})^Tv
+(a^{\forall})^T(u-v)
+(b^{\forall}+\ovr{b}^{\exists})^Tw&\leq-1,\\
u,v,w&\geq0.
\end{align*}
Now, there is the point where the equivalence cannot be easily establish. We can conclude that for each $z\in\{\pm1\}^{n'}$ the interval system
\begin{align*}
(\umace{A}^{\exists})^Tu-(\omace{A}^{\exists})^Tv
+(\imace{A}^{\forall})^T(u-v)
+(\imace{C}^{\forall}+\umace{C}^{\exists})^Tw&\geq0,\\
(\Mid{B}^{\exists}-\Rad{B}^{\exists}\diag(z))^Tu
-(\Mid{B}^{\exists}+\Rad{B}^{\exists}\diag(z))^Tv&\\
+(\imace{B}^{\forall})^T(u-v)
+(\imace{D}^{\forall}+\Mid{D}^{\exists}-\Rad{D}^{\exists}\diag(z))^Tw&=0,\\
(\ovr{a}^{\exists})^Tu-(\uvr{a}^{\exists})^Tv
+(\ivr{a}^{\forall})^T(u-v)
+(\ivr{b}^{\forall}+\ovr{b}^{\exists})^Tw&\leq-1,\\
u,v,w&\geq0.
\end{align*}
is weakly solvable. By Theorem~\ref{thmHladik}, for each such system there is $s\in\{\pm1\}^m$ such that
\begin{align*}
(\umace{A}^{\exists})^Tu-(\omace{A}^{\exists})^Tv
+(\Mid{A}^{\forall}+\diag(s)\Rad{A}^{\forall})^T(u-v)
+(\omace{C}^{\forall}+\umace{C}^{\exists})^Tw&\geq0,\\
(\Mid{B}^{\exists}-\Rad{B}^{\exists}\diag(z))^Tu
-(\Mid{B}^{\exists}+\Rad{B}^{\exists}\diag(z))^Tv&\\
+(\Mid{B}^{\forall}+\diag(s)\Rad{B}^{\forall})^T(u-v)
+(\omace{D}^{\forall}+\Mid{D}^{\exists}-\Rad{D}^{\exists}\diag(z))^Tw&\geq0,\\
-(\Mid{B}^{\exists}-\Rad{B}^{\exists}\diag(z))^Tu
+(\Mid{B}^{\exists}+\Rad{B}^{\exists}\diag(z))^Tv&\\
-(\Mid{B}^{\forall}-\diag(s)\Rad{B}^{\forall})^T(u-v)
-(\umace{D}^{\forall}+\Mid{D}^{\exists}-\Rad{D}^{\exists}\diag(z))^Tw&\geq0,\\
(\ovr{a}^{\exists})^Tu-(\uvr{a}^{\exists})^Tv
+(\Mid{a}^{\forall}-\diag(s)\Rad{a}^{\forall})^T(u-v)
+(\uvr{b}^{\forall}+\ovr{b}^{\exists})^Tw&\leq-1,\\
u,v,w&\geq0.
\end{align*}
is solvable. By the Farkas lemma again, the system
\begin{align*}
(\umace{A}^{\exists}+\Mid{A}^{\forall}+\diag(s)\Rad{A}^{\forall})x
+(\Mid{B}^{\exists}-\Rad{B}^{\exists}\diag(z)
  +\Mid{B}^{\forall}+\diag(s)\Rad{B}^{\forall})y^1&\\
-(\Mid{B}^{\exists}-\Rad{B}^{\exists}\diag(z)
  +\Mid{B}^{\forall}-\diag(s)\Rad{B}^{\forall})y^2
&\leq\ovr{a}^{\exists}+\Mid{a}^{\forall}-\diag(s)\Rad{a}^{\forall},\\
-(\omace{A}^{\exists}+\Mid{A}^{\forall}+\diag(s)\Rad{A}^{\forall})^Tx
-(\Mid{B}^{\exists}+\Rad{B}^{\exists}\diag(z)
  +\Mid{B}^{\forall}+\diag(s)\Rad{B}^{\forall})y^1&\\
+(\Mid{B}^{\exists}+\Rad{B}^{\exists}\diag(z)
  +\Mid{B}^{\forall}-\diag(s)\Rad{B}^{\forall})y^2
&\leq-\uvr{a}^{\exists}-\Mid{a}^{\forall}+\diag(s)\Rad{a}^{\forall},\\
(\omace{C}^{\forall}+\umace{C}^{\exists})x
+(\omace{D}^{\forall}+\Mid{D}^{\exists}-\Rad{D}^{\exists}\diag(z))y^1
-(\umace{D}^{\forall}+\Mid{D}^{\exists}-\Rad{D}^{\exists}\diag(z))y^2
&\leq \uvr{b}^{\forall}+\ovr{b}^{\exists}\\
x,y^1,y^2&\geq0
\end{align*}
is not solvable, which completes the proof.
\end{proof}

Example~\ref{exFail} from Section~\ref{ssAeSoltySols} can also be used here to illustrate the situation that the condition from Proposition~\ref{propAeSoltySuff} is not necessary in general. Indeed, that interval system is AE solvable, but the system \nref{desAEsoltySuf}, which takes the form of
$$
-zy^1+zy^2\leq-2,\ 
y^1+y^2\leq1,\ 
y^1,y^2\geq0
$$
is solvable for no $z\in\{\pm1\}$.

%%%%%%%%%%%%%%%%%%%%%%%%%%%%%%%%%%%%%%%%%%%%%%%%%%%%%%%%%%%%%%% 
\subsection{Special cases of AE solvability}

Proposition~\ref{propAEsolty} generalizes several classical results on solvability of interval systems. In particular, for $\imace{A}\in\IR^{m\times n}$ and $\ivr{b}\in\IR^m$ we have:
\begin{itemize}
\item
For strong solvability of interval equations we obtain the same result as the characterization by Rohn \cite{Roh2003,Roh2006}. The interval system $\imace{A}x=\ivr{b}$ is strongly solvable if and only if the system 
\begin{align*}
(\Mid{A}+\diag(s)\Rad{A})x^1
-(\Mid{A}-\diag(s)\Rad{A})x^1
=\Mid{b}-\diag(s)\Rad{b},\ \ x^1,x^2\geq0
\end{align*}
is solvable for each $s\in\{\pm1\}^m$.
\item
For strong solvability of interval equations with nonnegative variables we obtain the same result as the characterization by Rohn \cite{Roh1981,Roh2006}. The interval system $\imace{A}x=\ivr{b}$, $x\geq0$ is strongly solvable if and only if the system 
\begin{align*}
(\Mid{A}+\diag(s)\Rad{A})x=\Mid{b}-\diag(s)\Rad{b},\ \ x\geq0
\end{align*}
is solvable for each $s\in\{\pm1\}^m$.
\item
For strong solvability of interval inequalities we obtain the same result as the characterization by Rohn \& Kreslov\'{a} \cite{RohKre1994,Roh2006}. The interval system $\imace{A}x\leq\ivr{b}$ is strongly solvable if and only if the system 
\begin{align*}
\omace{A}x^1-\umace{A}x^2\leq\uvr{b},\ \ x^1,x^2\geq0
\end{align*}
is solvable.
\item
For strong solvability of interval inequalities with nonnegative variables we obtain the same result as the classical characterization by Machost \cite{Mach1970,Roh2006}. The interval system $\imace{A}x\leq\ivr{b}$, $x\geq0$ is strongly solvable if and only if the system 
\begin{align*}
\omace{A}x\leq\uvr{b},\ \ x\geq0
\end{align*}
is solvable.
\item
For weak solvability of interval equations with nonnegative variables we obtain the same result as the consequence of the classical characterization by Oettli \& Prager \cite{OetPra1964,Roh2006}. The interval system $\imace{A}x=\ivr{b}$, $x\geq0$ is weakly solvable if and only if the system 
\begin{align*}
\umace{A}x\leq\ovr{b},\ \ \omace{A}x\geq\uvr{b},\ \ x\geq0
\end{align*}
is solvable.
\item
For weak solvability of interval inequalities with nonnegative variables we obtain also the same result as the well known characterization; see, e.g., \cite{Roh2006}. The interval system $\imace{A}x\leq\ivr{b}$, $x\geq0$ is weakly solvable if and only if the system 
\begin{align*}
\umace{A}x\leq\ovr{b},\ \ x\geq0
\end{align*}
is solvable.
\end{itemize}

Unfortunately, weak solvability of interval systems with free variables is not involved in our generalization. Besides weak and strong solvability, we have also the following analogies of tolerable and controllable solvabilities by Li et al.~\cite{LiLuo2014s} as simple consequences of Proposition~\ref{propAEsolty}:
\begin{itemize}
\item
For each $A\in\imace{A}$ there is $b\in\ivr{b}$ such that $Ax=b$ is solvable if and only if
\begin{align*}
\uvr{b}\leq(\Mid{A}+\diag(s)\Rad{A})x^1
-(\Mid{A}-\diag(s)\Rad{A})x^2
\leq\ovr{b},\ \ x^1,x^2\geq0
\end{align*}
is solvable for each $s\in\{\pm1\}^m$.
\item
For each $A\in\imace{A}$ there is $b\in\ivr{b}$ such that $Ax=b$, $x\geq0$ is solvable if and only if
\begin{align*}
\uvr{b}\leq(\Mid{A}+\diag(s)\Rad{A})x
\leq\ovr{b},\ \ x\geq0
\end{align*}
is solvable for each $s\in\{\pm1\}^m$.
\item
For each $b\in\ivr{b}$ there is $A\in\imace{A}$ such that $Ax=b$, $x\geq0$ is solvable if and only if
\begin{align*}
\umace{A}x\leq \Mid{b}-\diag(s)\Rad{b}\leq\omace{A}x,\ \ x\geq0
\end{align*}
is solvable for each $s\in\{\pm1\}^m$.
\item
For each $A\in\imace{A}$ there is $b\in\ivr{b}$ such that $Ax\leq b$ is solvable if and only if
\begin{align*}
\omace{A}x^1-\umace{A}x^2\leq\ovr{b},\ \ x^1,x^2\geq0
\end{align*}
is solvable.
\item
For each $A\in\imace{A}$ there is $b\in\ivr{b}$ such that $Ax\leq b$, $x\geq0$ is solvable if and only if
\begin{align*}
\omace{A}x\leq\ovr{b},\ \ x\geq0
\end{align*}
is solvable.
\item
For each $b\in\ivr{b}$ there is $A\in\imace{A}$ such that $Ax\leq b$, $x\geq0$ is solvable if and only if
\begin{align*}
\umace{A}x\leq\uvr{b},\ \ x\geq0
\end{align*}
is solvable.
\end{itemize}

%%%%%%%%%%%%%%%%%%%%%%%%%%%%%%%%%%%%%%%%%%%%%%%%%%%%%%%%%%%%%%% 
\section{Conclusion}

We characterized AE solutions and for a certain sub-class of problems we also characterized AE solvability. For general problems, we presented only a sufficient condition for AE solvability. A complete characterization of AE solvability remains an open problem.

%\subsubsection*{Acknowledgments.} 
%The author was supported by the Czech Science Foundation Grant P402/13-10660S.

%%%%%%%%%%%%%%%%%%%%%%%%%%%%%%%%%%%%%%%%%%%%%%%%%%%%%%%%%%%%%%% 
% REFERENCES
%%%%%%%%%%%%%%%%%%%%%%%%%%%%%%%%%%%%%%%%%%%%%%%%%%%%%%%%%%%%%%% 

\bibliographystyle{abbrv}
\bibliography{lin_sys_ae}

\end{document}